\documentclass[reqno,tbtags]{amsproc}
\usepackage{graphicx,amscd}
\usepackage{calrsfs}
\usepackage{amssymb}

\newtheorem{thm}{Theorem}[section]
\newtheorem{lem}[thm]{Lemma}
\newtheorem{cor}[thm]{Corollary}

\theoremstyle{remark}
\newtheorem{rem}[thm]{Remark}

\newcommand{\la}{{\lambda}}
\newcommand{\si}{{\sigma}}
\newcommand{\wt}{\tilde}

\begin{document}





\begin{center}

\renewcommand{\thefootnote}{\fnsymbol{footnote}}
{\Large \bf Description of the characters and factor representations
of the infinite symmetric inverse semigroup%
\footnote{Partially supported by the RFBR grants 08-01-00379-a and 09-01-12175-ofi-m.}.}\\[0.5in]
\vspace{2pt}%
\setcounter{footnote}{0}

{\large\bf A.~M.~Vershik, P.~P.~Nikitin}
\\[6pt]
St.~Petersburg Department \\
of the Steklov Mathematical Institute\\
27, Fontanka, 191023 St.Petersburg, Russia \\[4pt]
E-mail: {vershik@pdmi.ras.ru, pnikitin0103@yahoo.co.uk}
\end{center}

\begin{abstract}
We give a complete list of indecomposable characters of the infinite
symmetric semigroup. In comparison with the analogous list for the
infinite symmetric group, one should introduce only one new parameter,
which has a clear combinatorial meaning. The paper relies on the
representation theory of the finite symmetric semigroups and the
representation theory of the infinite symmetric group.
\end{abstract}

\section*{Introduction}

In this paper, we describe the characters of the infinite
symmetric semigroup. The main result establishes a link between the
representation theory of the finite symmetric semigroups developed by
Munn \cite{Munn_symm-inverse-semigroup}, \cite{Munn_inverse=semisimple},
Solomon \cite{Solomon},
Halverson \cite{Halverson}, Vagner \cite{Vagner}, Preston
\cite{Preston}, and Popova \cite{Popova} on the one hand, and the representation
theory of locally finite groups (in particular, the infinite symmetric
group) and locally semisimple algebras developed in the papers by Thoma
\cite{Thoma}, Vershik and Kerov
\cite{VK}--\cite{VK_survey2}, \cite{VK_survey} on the other hand.
The below analysis of the Bratteli diagram for the infinite
symmetric semigroup reminds the analogous analysis in the
more complicated case of describing the characters of the
Brauer--Weyl algebras \cite{VershikNikitin}. The symmetric semigroup
appeared not only in the literature on the theory of semigroups and their
representations, but also in connection with the representation theory of the
infinite symmetric group \cite{Olshansky} and the definition of the braid
semigroup \cite{Vershinin}; $q$-analogs of the symmetric semigroup were
also considered \cite{Halverson}. Apparently, the definition of the
infinite symmetric semigroup given in this paper, as well as problems
related to representations of this semigroup, have not yet been discussed in the
literature.

Consider the set of all one-to-one
\textit{partial} transformations of the set
$\{1,\dots,n\}$, i.e., one-to-one maps from a subset of
$\{1,\dots,n\}$ to a subset (possibly, different from the first one) of
$\{1,\dots,n\}$. We define the product of such maps as their
composition where it is defined. Thus we obtain a
semigroup with a zero (the map with the empty domain of definition),
which is usually called the \textit{symmetric inverse semigroup};
denote it by $R_n$
(there are also other notations, see \cite{CliffordPreston}, \cite{Solomon}).

Obviously, the symmetric group $S_n$ is a subgroup of the semigroup
$R_n\colon S_n\subset
R_n$. Further, $R_n$ can be presented as the semigroup of all
$0$-$1$ matrices with at most one $1$ in each row and each column
equipped with
matrix multiplication. This realization is similar to the natural
representation of the symmetric group. The matrices of this form are
in a one-to-one correspondence with all possible placements of
nonattacking rooks on the
$n\times n$ chessboard, that is why Solomon called this monoid (the
semigroup with a zero)
the \textit{rook monoid}.

The following properties of inverse semigroups and, in particular, the
symmetric inverse semigroup are of great importance (see the Appendix).

(1) the complex semigroup algebra of every finite inverse semigroup is
semisimple
(\cite{Oganesyan}, \cite{Munn_inverse=semisimple});

(2) every finite inverse semigroup can be isomorphically embedded into
a sym\-met\-ric  inverse semigroup
(\cite{Vagner}, \cite{Preston});

(3) the class of finite inverse semigroups generates exactly the class
of involutive semisimple bialgebras
\cite{Vershik_bialgebga}.

The following result, which describes the characters of a finite
inverse semigroup, was essentially discovered by several authors; its
combinatorial and dynamical char\-ac\-ter\-iza\-tion is given in
\cite{Halverson}.

The set of irreducible representations (and, consequently, the set of
irreducible characters) of the symmetric semigroup
$R_n$ is indexed by the set of all Young diagrams with at most
$n$ cells. The branching of representations in terms of diagrams looks
as follows: when passing from an irreducible representation of
$R_n$ to rep\-re\-sen\-ta\-tions of $R_{n+1}$, the corresponding Young diagram
either does not change, or obtains one new cell (grows).

The infinite symmetric group
$S_{\infty}$ is the countable group of all finitary (i.e.,
nonidentity only on a finite subset) one-to-one transformations of a countable
set. In the same way one can define the
\textit{infinite symmetric inverse semigroup}\footnote{Usually we omit the
word ``inverse'' and speak about the (infinite) symmetric
semigroup.} $R_\infty$ as the set of partial one-to-one transformations
of a countable set that are
nonidentity only on a finite subset.\footnote{Thus the infinite
symmetric inverse semigroup does not contain the zero map, since every
element must be identity on the complement of a finite set.}
The group $S_{\infty}$ is the inductive limit of the chain
$S_n$, $n=1,2,\dots$, with the natural embeddings of groups. In the
same way, the semigroups $R_n$, $n=1,2,\dots$, form a chain with
respect to the natural monomorphisms of semigroups\footnote{Under the
monomorphism $R_n\subset
R_{n+1}$, the zero of $R_n$ is mapped not to a zero, but to a certain
projection; more exactly, to the generator
$p_n\in R_{n+1}$, see Theorem~\ref{thm:gens&rels_Halverson}.}
$R_0\subset R_1\subset\dots\subset
R_n\subset\dots$, and its inductive limit is the infinite inverse
symmetric semigroup. The connection between the Bratteli diagram of
the infinite symmetric group (the Young graph) and that of the infinite
symmetric inverse semigroup leads naturally to introducing a new operation on
graphs, which associates with every Bratteli diagram its
``slow'' version. (Cf.\ the notion of the ``pascalization'' of a graph
introduced in
\cite{VershikNikitin}.)

Our results rely on the well-developed representation theory of the
infinite sym\-met\-ric group $S_\infty$ and, to some extent, generalize it.
Recall that the list of characters of the
infinite symmetric group was found by Thoma
\cite{Thoma}. The new proof of Thoma's theorem suggested by Vershik
and Kerov \cite{VK} was based on approximation of char\-ac\-ters of the
infinite symmetric group by characters of
finite symmetric groups and used the combinatorics of Young diagrams,
which, as is well known, parameterize the irreducible complex representations
of the finite symmetric groups. The parameters of indecomposable
characters in the exposition of \cite{VK} are interpreted as the
fre\-quen\-cies of the rows
and columns of a sequence of growing Young diagrams.
\textit{The main result of this paper is that the list of parameters
for the characters of the
infinite symmetric group is obtained from the list of Thoma
parameters by adding a new number from the interval
$[0,1]$}. The meaning of this new parameter is as follows. The
irreducible representations of the finite symmetric
semigroup $R_n$ are also parameterized by Young diagrams, but with an arbitrary number
of cells $k$ not exceeding $n$; thus,
apart from the limiting frequencies of rows and columns, a sequence of growing diagrams has another
parameter: the limit of the ratio
$k/n$, which is the relative velocity with which the corresponding
path passes through the levels of the
branching graph; or, in other words, the
deceleration of the rate of approximation of a character of the
infinite semigroup by characters of finite semigroups.

The description of the characters allows us to
construct a realization of the corresponding representations. They live in the same
space as the corresponding representations of the infinite symmetric group.
More
exactly, the space of the representation is constructed in exactly the same way as in the
model of factor representations of the infinite symmetric group
suggested in
\cite{VK_factor}, but with the extended list of parameters, see
Theorem~\ref{thm:realization}.

In the first section, we give the necessary background on the representation
theory of the finite symmetric inverse semigroups. The second section
is devoted to the representation
theory of the infinite symmetric semigroup
$R_\infty$ and contains our main results. In Appendix we collect
general information about finite inverse semigroups and some new facts
about their semigroup algebras regarded as Hopf algebras.

\section{The representation theory of the finite symmetric inverse semigroups}

\subsection{The semisimplicity of the semigroup algebra ${\mathbb{C}}[R_n]$. 
The complete list of irreducible representations}

 We define the  \textit{rank} of a map $a\in R_n$
as the number of elements on which this map is not defined. Each of
the sets $A_r = \{a\in R_n\mid$ the rank of $a$ is at least $r\}$
for $0 \le r \le n$ is an ideal of the semigroup
$R_n$. The chain of ideals
$$
R_n = A_0 \supset A_1 \supset \dots \supset A_n
$$
is a principal series of the semigroup
$R_n$, i.e., there is no ideal lying strictly between
$A_r$ and
$A_{r+1}$, see Theorem~\ref{thm:CC[R_n] isom oplus}.

Denote by ${\mathbb{C}}[S_n]$ the complex group algebra of the symmetric group
$S_n$. This algebra, as well as the group algebra of every finite group, is
semisimple, since in it there exists an invariant inner product.

The complex semigroup algebra of an inverse group is always semisimple
too, as follows from the general Theorem~\ref{thm:inverse_semisimple}.
An explicit decomposition of the algebra ${\mathbb{C}}[R_n]$ into matrix
components was suggested by Munn \cite{Munn_symm-inverse-semigroup}.

\begin{thm}[Munn]\label{thm:CC[R_n] isom oplus}
The algebra ${\mathbb{C}}[R_n]$ is semisimple and has the form
$$
{\mathbb{C}}[R_n] \cong \bigoplus_{r=0}^n M_{\binom{n}{r}} ({\mathbb{C}}[S_r]).
$$
\end{thm}

Here $M_l(A)$ is the algebra of matrices of order $l$ over an algebra $A$.
A description of the representations of the algebra
${\mathbb{C}}[R_n]$ is given by the following theorem.

\begin{thm}[Munn]
Let $S$ be a semigroup isomorphic to the semigroup $M_n(G)$ of $n\times n$
matrices with elements from a group $G$. Let $F$ be a field whose
characteristic is equal to zero or is a prime not dividing the order of $G$.
Let $\{\gamma_p\}_{p=1}^k$ be the complete list of non\-equiv\-a\-lent
irreducible representations of the group $G$ over $F$.
Denote by $\gamma_p'$ the map given by the formula
$$
\gamma_p'(\{x_{ij}\}) = \{\gamma_p(x_{ij})\}
$$
for every matrix $\{x_{ij}\}\in S=M_n(G)$. Then
$\{\gamma_p'\}_{p=1}^k$ is the complete list of non\-equiv\-a\-lent
irreducible representations of the semigroup
$S$ over $F$.
\end{thm}

Denote by $\mathcal{P}_r$ the set of all partitions of a positive
integer $r$. It follows from the previous theorem that the set of
irreducible representations of the semigroup $R_n$ can be naturally
indexed by the set $\bigcup_{r=0}^n \mathcal{P}_r$.

\begin{rem}\label{rem:ind-repres}
As can be seen from the form of irreducible representations of the
semigroup $R_n$ described above, each such representation is an
extension of a uniquely defined induced representation of the group
$S_n$. More exactly, for the irreducible representation of
$R_n$ corresponding to a partition
$\lambda\in\mathcal{P}_r$, consider the representation of the subgroup
$S_r\times S_{n-r}\subset S_n$ in which the action of
$S_r$ corresponds to the partition $\lambda$ and
$S_{n-r}$ acts trivially. The corresponding induced representation of
$S_n$ can be extended to the original irreducible representation of
$R_n$. (This was also observed in~\cite{Olshansky}.)
\end{rem}

\begin{rem}\label{rem:involution}
On the semigroup algebra ${\mathbb{C}}[R_n]$
of the symmetric semigroup, as well as on the group algebra ${\mathbb{C}}[S_n]$
of the symmetric group, there is an involution, which, in particular, sends every
irreducible representation $\pi$ to the representation sgn$\pi$.
It corresponds to the natural involution on the Young
graph and, consequently, of the slow Young graph (for the definition,
see Section~2.1) that sends a diagram to its reflection
in the diagonal. However, it is not
an involution of the group $S_n$ or the semigroup $R_n$.
\end{rem}

\subsection{A formula for the characters of the finite symmetric semigroup}
Munn \cite{Munn_symm-inverse-semigroup} also found a formula that
expresses the characters of the symmetric inverse semigroup in terms
of characters of the symmetric groups. In order to state the corresponding
theorem, for every subset
$K\subset \{1,\dots,n\}$, $|K| = r$, fix an arbitrary partial bijection
$\mu_K\colon K\mapsto\{1,\dots,r\}$. By
$\mu_K^{-}\colon\{1,\dots,r\}\mapsto K$ we denote the map
inverse to $\mu_K$ on $K$; thus $\mu_K^{-} \circ \mu_K$ is the identity map
on the set $K$.

\begin{thm}[Munn]\label{thm:chi_R-n}
Let $\chi^*$ be the character of the irreducible representation of the
semigroup $R_n$ corresponding to a partition
$\lambda\in\mathcal{P}_r$, $1\le r\le n$. Let $\chi$ be the corresponding
character of the symmetric group
$S_r$. Then for every element $\si\in R_n$,
$$
\chi^*(\si) = \sum \chi(\mu_K\si\mu_K^{-}),
$$
where the sum is taken over all subsets $K$ of the domain of
definition of
$\si$ such that $|K|=r$ and $K\si=K$.
\end{thm}

\subsection{Presentations of the semigroup $R_n$ by generators and relations} 
We are interested in families of generators of the semigroups
$\{R_n\}_{n=0}^\infty$ that increase under the embeddings $R_n\subset
R_{n+1}$. This condition is satisfied for the generators suggested by
Popova \cite{Popova} and those suggested by Halverson
\cite{Halverson}. In Halverson's paper, the generators and relations
are described for a $q$-analog of the symmetric inverse semigroup. Below
we present the particular case of his result for
$q=1$.

Let $\si_i$, $1 \le i < n$, be the Coxeter generators of the group
$S_n$. By $p_i\in R_n$,
$1\le i\le n$, we denote the following maps: $p_i(j)$ is not defined if
$j\le i$, and $p_i(j) =
j$ if $j > i$.

\begin{thm}[Popova]\label{thm:gens&rels_Popova}
The semigroup $R_n$ is generated by the elements
$\si_1$, \dots, $\si_{n-1}$, $p_1$ with the following relations:

(1) the Coxeter relations for the group $S_n$;

(2) $\si_2 p_1 \si_2 = \si_2\si_3\cdots\si_{n-1} p_1 \si_2\si_3\cdots\si_{n-1} = p_1 = p_1^2$;

(3) $(p_1\si_1)^2 = p_1\si_1 p_1 = (\si_1 p_1)^2$.
\end{thm}

\begin{thm}[Halverson]\label{thm:gens&rels_Halverson}
The semigroup $R_n$ is generated by the elements
$\si_1,\dots,\allowbreak\si_{n-1},p_1,\dots,p_n$
with the following relations:

(1) the Coxeter relations for the group $S_n$;

(2) $\si_i p_j = p_j \si_i = p_j$ for $1\le i < j \le n$;

(3) $\si_i p_j = p_j \si_i$ for $1 \le j < i \le n-1$;

(4) $p_i^2 = p_i$ for $1 \le i \le n$;

(5) $p_{i+1} = p_i \si_i p_i$ for $1 \le i \le n-1$.
\end{thm}

An  interesting presentation of the semigroup
$R_n$ by generators and relations was suggested by Solomon
\cite{Solomon}: in addition to the Coxeter generators of the group
$S_n$, he considers also the ``right shift'' $\nu$ defined as
$$
\nu(i) = \begin{cases}
i+1 & \text{for $1\le i < n$,}\\
\textrm{is not defined} & \text{for $i=n$.}
\end{cases}
$$

\begin{thm}[Solomon]
The semigroup $R_n$ is generated by the elements
$\si_1$, \dots, $\si_{n-1}$, $\nu$ with the following relations:

(1) the Coxeter relations for the group $S_n$;

(2) $\nu^{i+1}\si_i = \nu^{i+1}$;

(3) $\si_i\nu^{n-i+1} = \nu^{n-i+1}$;

(4) $\si_i\nu = \nu\si_{i+1}$;

(5) $\nu\si_1\si_2\si_3\cdots\si_{n-1}\nu = \nu$,

\noindent where $1\le i\le n-1$ in (1)--(3) and (5), and $1\le i\le
n-2$ in (4).
\end{thm}

\section{The representation theory of the infinite symmetric inverse semigroup}

In this section we assume that the reader is familiar with the basic
notions and results of the theory of locally semisimple and AF
algebras. Besides, we use some facts from the
representation theory of the finite symmetric
groups $S_n$ and the infinite symmetric group $S_\infty$.
See, e.g., \cite{VK_survey}.

There is a natural embedding $R_n\subset R_{n+1}$ of semigroups
under which every map from
$R_n$ goes to a map from $R_{n+1}$ that sends the element $n+1$ to
itself.
Consider the inductive limit of the chain
$R_0\subset R_1\subset\dots\subset R_n\subset\dots$ of semigroups, which we will
call the \textit{infinite symmetric inverse
semigroup} $R_\infty$.

\subsection{The branching graph of the algebra $\boldsymbol{{\mathbb{C}}[R_\infty]}$}
Let $\mathbb{Y}$ be the Young graph, and let
$\mathbb{Y}_n$ be the level of $\mathbb{Y}$ whose vertices are indexed by
all partitions of the integer $n$ (Young diagrams with $n$ cells).
By $|\la|$ we denote the number of cells in a diagram
$\la$ (the sum of the parts of the partition $\la$).

Denote by $\wt{\mathbb{Y}}$ the branching graph of the semigroup algebra
${\mathbb{C}}[R_\infty]$. It
was described by Halverson \cite{Halverson}.

\begin{thm}[Halverson]
The branching graph $\wt{\mathbb{Y}}$ can be described as follows:

(1) the vertices of the $n$th level are indexed by all Young diagrams
with at most $n$ cells:
$\wt{\mathbb{Y}}_n=\bigcup_{i=0}^n\mathbb{Y}_i$;

(2) vertices $\lambda\in\wt{\mathbb{Y}}_n$ and $\mu\in\wt{\mathbb{Y}}_{n+1}$
are joined by an edge if either
$\lambda=\mu$ or $\mu$ is obtained from $\lambda$ by adding one cell.
\end{thm}

This leads us to the following definition of the
\textit{slow graph}~$\wt\Gamma$ constructed from a branching graph
$\Gamma$:

(1) the set of vertices of the  $n$th level of $\wt\Gamma$ is the
union of the sets of vertices of all levels of the original graph
$\Gamma$ with indices at most $n$, i.e., $\wt\Gamma_n=\bigcup_{i=0}^n\Gamma_i$;

(2) vertices $\lambda\in\wt\Gamma_n$ and $\mu\in\wt\Gamma_{n+1}$
are joined by an edge if either
$\lambda=\mu$ or $\mu$ is joined by an edge with $\lambda$ in the
original graph.

Recall the definition of the Pascal graph $\mathbb{P}$:

(1) the set $\mathbb{P}_n$ of vertices of the $n$th level consists of
all
pairs of integers $(n, k)$, $0\le k\le n$;

(2) vertices $(n, k)\in\mathbb{P}_n$ and $(n+1,l)\in\mathbb{P}_{n+1}$
are joined by an edge if either $l=k$ or
$l=k+1$.

Observe that if the original graph $\Gamma$ is the chain
(the graph whose each level consists of a single vertex), then
the corresponding slow graph
$\wt\Gamma$ coincides with $\mathbb{P}$. By analogy with the Pascal graph, we
index the vertices of the $n$th level $\wt\Gamma_n$
of the slow graph with the pairs
$(n,\la)$, where $\la\in\Gamma_i$, $i\le n$.

\begin{rem}
Note that if $G=\mathbb{P}$ is the Pascal
graph, then the corresponding slow graph  $\wt{G}$ is the three-dimensional analog
of the Pascal graph. For the three-dimensional Pascal graph, the slow
graph is the four-dimensional Pascal graph, etc. For the definition of
the multidimensional analogs of the Pascal graph and a description of
the traces of the corresponding algebras, see, e.g.,
\cite{VK_survey2}.
\end{rem}

\begin{rem}
The set of paths on the branching graph
$\wt{\mathbb{Y}}$ is in bijection with the random walks on $\mathbb{Y}$
of the following form: at each moment, we are allowed
either to stay at the same vertex or to descend to the previous level
in an admissible way. In view of this description, graphs similar to
$\wt{\mathbb{Y}}$ are called slow.
\end{rem}

\begin{rem}
In \cite{VershikNikitin}, the representation theory of the infinite
Brauer algebra was studied. As in the previous remark, one can
construct a bijection between the paths on the branching graph of the
Brauer algebra and the random walks of a similar form on the Young
graph: starting from the empty diagram, at each
step we can move either to a vertex of the next level (joined by an
edge with the current vertex) or to a vertex of the previous level (joined
by an  edge with the current vertex).
\end{rem}

\subsection{Facts from the theory of locally semisimple algebras}
Given a branch\-ing graph $\Gamma$,
denote by $T(\Gamma)$ the space of paths of
$\Gamma$. On $T(\Gamma)$ we have the ``tail'' equivalence relation
(see~\cite{VK}): paths $x,y\in T(\Gamma)$ are equivalent, $x\sim y$,
if they coincide from some level on. The partition of $T(\Gamma)$
into the equivalence
classes will be denoted by
$\xi = \xi_\Gamma$. Also, for every $k\in\mathbb{N}\cup{0}$ and every
path $s=(s_0,
s_1,\dots, s_k)$ of length $k$, denote by $F_s\subset T(\Gamma)$ the cylinder $F_s = \{t\in T\mid t_i =
s_i\,\text{ for }\,0\le i\le k\}$.

Given $x,y\in\Gamma$, by $\dim(x;y)$ denote the number of paths leading
from $x$ to $y$. By
$\dim(y) = \dim(\varnothing;y)$ denote the total number of paths
leading to
$y$. By $\mathcal{E}(\Gamma)$ denote the set of ergodic central
measures on $T(\Gamma)$. Given
$\mu\in\mathcal{E}(\Gamma)$ and a vertex $y$, by $\mu(y)$ denote the measure of the set
of all paths passing through
$y$, i.e., the total measure of all cylinders
$F_s$, $s=(s_0, s_1, \dots, s_{|y|})$, $s_{|y|}=\nobreak y$.

We will use the following description of the characters of a locally
semisimple algebra and the central measures on its branching graph
(ergodic method).

\begin{thm}[\cite{VK}]\label{thm:center_measure}
For every central ergodic measure $\mu$, the set of paths
$s = (s_0, s_1, \dots, s_f, \dots)$
such that
$$
\mu(y) = \lim_{f\to\infty}\frac{\dim(y)\cdot\dim(y;s_f)}{\dim s_f}
$$
 for all vertices $y$ is of full measure.
\end{thm}

\begin{thm}[\cite{VK}]\label{thm:approx}
For every character $\phi$ of the algebra
$A = C^*(\bigcup_{f=0}^\infty A_f)$, there exists a path
$\{\la_f\}_{f=0}^\infty$ in the Bratteli diagram such that
$$
\phi(a) = \lim_{f\to\infty}\frac{\chi_{\la_f}(a)}{\dim \la_f}
$$
for all $a\in A$. Here $\chi_{\la_f}$ is the character of the
representation $\la_f$ of the algebra $A_f$ and
$\dim\la_f$ is its dimension.
\end{thm}

\subsection{Description of the central measures on slow graphs}
The key property of an arbitrary slow graph
$\wt\Gamma$ is that we can present the space of paths
$T(\wt\Gamma)$ as the direct product of the spaces of paths
$T(\Gamma)$ and $T(\mathbb{P})$. The same is true for the sets of paths
between any two vertices. Moreover, the partition
$\xi_{\wt\Gamma}$ and the central ergodic measures on
$T(\wt\Gamma)$ can also be presented as corresponding products.

\begin{lem}\label{lem:paths-direct-product}
Let $\Gamma$ be the branching graph of a locally semisimple algebra and
$\wt\Gamma$ be the corresponding slow graph. Then

1. $T(\wt\Gamma) = T(\Gamma) \times T(\mathbb{P})$. Moreover, the
number of paths between any two vertices of the slow graph
$\wt\Gamma$ is the product of the number of paths between the
corresponding vertices of the original graph
$\Gamma$ and the number of vertices between the corresponding vertices
of the Pascal graph $\mathbb{P}$:
\begin{equation}\label{dim=prod}
\dim_{\wt\Gamma}((n_1,\la_1);(n_2,\la_2)) = \dim_\Gamma(\la_1,\la_2) \cdot
\dim_\mathbb{P}((n_1,|\la_1|);(n_2,|\la_2|)).
\end{equation}

2. Let $s_{\wt\Gamma}, t_{\wt\Gamma} \in T(\wt\Gamma)$, $s_\Gamma, t_\Gamma \in T(\Gamma)$,
$s_\mathbb{P}, t_\mathbb{P} \in T(\mathbb{P})$, and let $s_{\wt\Gamma}$
correspond to the pair $(s_\Gamma,
s_\mathbb{P})$ and $t_{\wt\Gamma}$ correspond to the pair
$(t_\Gamma, t_\mathbb{P})$. Then $s_{\wt\Gamma}
\sim t_{\wt\Gamma}$ (with respect to $\xi_{\wt\Gamma}$) if and only if
$s_\Gamma \sim
t_\Gamma$ (with respect to  $\xi_\Gamma$) and
$s_\mathbb{P} \sim t_\mathbb{P}$ (with respect to
$\xi_\mathbb{P}$).
\end{lem}

\begin{proof}[Proof]
1. To each path in the graph $\wt\Gamma$ there corresponds a unique
\textit{strictly} increasing sequence of vertices of the original graph
$\Gamma$. Moreover, to each path ${(i,
\la_i)}_{i=n_1}^{n_2}$ in the graph $\wt\Gamma$ we can associate the
path ${(i, |\la_i|)}_{i=n_1}^{n_2}$ in the Pascal graph. It is easy to
see that the original path is uniquely determined by the constructed
pair of paths, whence $T(\wt\Gamma) = T(\Gamma) \times T(\mathbb{P})$.

Note that the constructed map determines a bijection between the paths
from a vertex $(n_1,\la_1)$ to a vertex
$(n_2,\la_2)$ in the graph $\wt\Gamma$ and the pairs of paths
between the corresponding vertices in the original graph
$\Gamma$ and in the Pascal graph
$\mathbb{P}$, which proves formula \eqref{dim=prod}.

2. The bijection in the proof of Claim~1 is constructed in such a way
that the tail of a path
$t_{\wt\Gamma} = (t_\Gamma,
t_\mathbb{P})$  depends only
on the tails of the paths
$t_\Gamma$ and $t_\mathbb{P}$, and vice versa. \qed
\end{proof}

\begin{thm}[Description of the central measures]\label{thm:measures-direct-product}
There is a natural bijection $\mathcal{E}(\wt\Gamma)\cong \mathcal{E}(\Gamma) \times
\mathcal{E}(\mathbb{P})$. Every central ergodic measure $M_{\wt\Gamma} \in
\mathcal{E}(\wt\Gamma)$ is the product of central ergodic measures $M_\Gamma\in
\mathcal{E}(\Gamma)$ and $M_\mathbb{P} \in \mathcal{E}(\mathbb{P})$;
namely, $M_{\wt\Gamma}(F_{(n,
\la)}) = M_\Gamma(F_\la) \cdot M_\mathbb{P}(F_{(n,|\la|)})$ for every
cylinder $F_{(n, \la)}$.
\end{thm}

\begin{proof}[Proof]
In accordance with the decomposition
$T(\wt\Gamma) = T(\Gamma) \times T(\mathbb{P})$, given a central ergodic measure $M_{\wt\Gamma} \in
\mathcal{E}(\wt\Gamma)$, consider the
projections $M_\Gamma\in
\mathcal{E}(\Gamma)$ and $M_\mathbb{P} \in \mathcal{E}(\mathbb{P})$
defined as follows:
$$
M_\Gamma(F_\la) = \sum_{n\ge|\la|} M_{\wt\Gamma} (F_{(n,\la)}), \qquad M_\mathbb{P}(F_{(n,k)}) =
\sum_{|\la| = k} M_{\wt\Gamma} (F_{(n,\la)}).
$$
The measures $M_\Gamma$ and $M_\mathbb{P}$ are central by the
centrality of $M_{\wt\Gamma}$.

Further, according to formula \eqref{dim=prod} from Lemma
\ref{lem:paths-direct-product} and Theorem
\ref{thm:center_measure},
\begin{align}
M_{\wt\Gamma} (F_{(n,\la)})&=
\lim_{f\to\infty} \frac {\dim((n,\la_n);(f,\la_f))} {\dim (f,\la_f)}\notag\\
&=\lim_{f\to\infty} \frac {\dim_\mathbb{P}((n,|\la_n|);(f,|\la_f|))} {\dim_\mathbb{P}(f,|\la_f|)}
\cdot
\frac {\dim_\Gamma(\la_n;\la_f)} {\dim_\Gamma(\la_f)}\notag\\
&=\lim_{f\to\infty} \frac {\dim_\mathbb{P}(n,|\la_n|);(f,|\la_f|))} {\dim_\mathbb{P}(f,|\la_f|)}
\cdot \lim_{f\to\infty} \frac {\dim_\Gamma(\la_n;\la_f)} {\dim_\Gamma(\la_f)}\,.\label{eqn:
tildeM=lim}
\end{align}
The limits in the right-hand side of
\eqref{eqn: tildeM=lim} exist and are equal to $M_\Gamma(F_\la)$ and
$M_\mathbb{P}(F_{(n,k)})$\kern-1pt, which proves the required formula
for $M_{\wt\Gamma}$. The ergodicity of the measures
$M_\Gamma$ and $M_\mathbb{P}$ follows from the ergodicity of the measure
$M_{\wt\Gamma}$.

Conversely, the product (in the above sense) of central ergodic measures
$M_\Gamma\in
\mathcal{E}(\Gamma)$ and $M_\mathbb{P} \in \mathcal{E}(\mathbb{P})$
is a central ergodic measure
$M_{\wt\Gamma} \in \mathcal{E}(\wt\Gamma)$. Its centrality follows from
Lemma~\ref{lem:paths-direct-product}, and its ergodicity follows from
equation~\eqref{eqn:
tildeM=lim}.\qed
\end{proof}

Recall (see, e.g.,~\cite{VK_survey2}) that for the Pascal graph
$\mathbb{P}$, the limits in Theorem
\ref{thm:center_measure} exist if and only if for the path
$$
((0,k_0),(1,k_1),\dots,(f,k_f),\dots)
$$
the limit
\begin{equation} \label{eqn: pi=lim}
\lim_{f\to\infty} k_f/f = \delta, \qquad \delta\in [0;1],
\end{equation}
does exist, and to every $\delta\in [0;1]$ there corresponds a unique
central measure $M_\mathbb{P} =
M_\mathbb{P}^\delta$.

\begin{cor}
Every measure $M_{\wt\Gamma}\in \mathcal{E}(\wt\Gamma)$ is
parameterized by a pair $(\delta, M_\Gamma)$,
$\delta\in [0;1]$, $M_\Gamma\in \mathcal{E}(\Gamma)$.
\end{cor}

\begin{cor}\label{cor:supp-al,be,p}
The measure $M_{\wt\Gamma} = (\delta, M_\Gamma)$ on $T(\wt\Gamma)$
is concentrated on paths for which the corresponding paths in the graph
$\Gamma$ lie in the support of the measure $M_\Gamma$ and, besides,
the limit \eqref{eqn: pi=lim} does exist.

In particular, consider an arbitrary central ergodic measure
$M_\mathbb{Y}$ on the graph
$\mathbb{Y}$ corresponding to parameters
$\alpha=\{\alpha_i\}$, $\beta=\{\beta_i\}$, $\gamma$. Then the measure
$M_{\wt{\mathbb{Y}}} = (\delta,M_\mathbb{Y})$ on $T(\wt{\mathbb{Y}})$
is concentrated on paths of the form
$\{(f,\la_f)\}$ for which the corresponding limits for the sequence
$\{\la_f\}$ are equal to
$\{\alpha_i\}$ and $\{\beta_i\}$ and, besides,
$\lim_{f\to\infty} |\la_f|/f=\delta$.
\end{cor}

\subsection{A formula for the characters of the infinite symmetric semigroup} 
The bijection described above between the sets of central
measures on the spaces of paths of the graph
$\Gamma$ and of the slow graph
$\wt\Gamma$ holds for an arbitrary graded graph
$\Gamma$. This bijection can be translated to the sets of characters
of the algebras corresponding to these graphs (see
Corollary~\ref{cor:bijection-chars} below) via the correspondence
between central measures and characters; however, explicit formulas
for characters substantially depend on the graphs and algebras and
have no universal meaning. Below we prove a formula that expresses a
character of the algebra
${\mathbb{C}}[R_\infty]$ in terms of the corresponding character of the algebra
${\mathbb{C}}[S_\infty]$. In this section, by a character we always mean an
indecomposable character.

\begin{cor}\label{cor:bijection-chars}
The parametrization of the set of central measures described above determines
a bijection which sends every pair
$(\delta,\chi^{S_\infty}_{\alpha,\beta,\gamma})$, where
$\delta\in[0,1]$ and
$\chi^{S_\infty}_{\alpha,\beta,\gamma}$ is a character of the algebra
${\mathbb{C}}[S_\infty]$, to the character
$\chi^{R_\infty}_{\alpha,\beta,\gamma,\delta}$ of the algebra ${\mathbb{C}}[R_\infty]$.
\end{cor}

To simplify the notation, below we often omit the superscripts and the
parameter $\gamma$ (which can be expressed in terms of $\alpha$ and $\beta$),
setting
$$
\chi_{\alpha,\beta} \equiv \chi^{S_\infty}_{\alpha,\beta,\gamma}, \qquad \chi_{\alpha,\beta,\delta}
\equiv \chi^{R_\infty}_{\alpha,\beta,\gamma,\delta}.
$$

The conjugation of an element $\si\in R_n$ by an element of the
symmetric group does not change the value of a character, so it
suffices to consider
\textit{reduced elements} $\si^\circ\in R_n$, for which all fixed
points are at the end: for every
$\si\in R_n$ there exist $g\in S_n$,
$n(\si)\in \mathbb{N}\cup 0$ such that $\si^\circ = g\si g^{-1}$ and
$\si^\circ(i)\neq i$ for
$i<n(\si)$ and $\si^\circ(i)=i$ for $i\ge n(\si)$. By the definition
of the embedding $R_n\subset R_{n+1}$, we may assume that
$\si^\circ\in R_{n(\si)}$. The order $n(\si)$ of the element $\si^\circ$
is uniquely determined by the element $\si$.

Let us introduce a set $M_k(\si)\subset S_n$ whose elements are
indexed by all $k$-element subsets
$K\subset\{1,\dots,n\}$ fixed under $\si$: to each such
subset we associate the bijection
$\tilde\si\in S_n$ that coincides with $\si$ on $K$ and is identity
at all other points.

Note that for every element $\si$ of the semigroup $R_n$ we may consider the maximal
(possibly, empty) subset of
$\{1,\dots,n\}$ that is mapped by $\si$ to itself in a
one-to-one manner. The restriction of $\si$ to this subset will
be called the \textit{invertible part} of $\si$. The
invertible part of every element
$\si\in R_n$ can be regarded as an element of some symmetric group
$S_r$, $r\le n$, and, consequently, it can be written as a product of
disjoint cycles. The set $M_k(\si)$ can also be parameterized by the
set of all subcollections of cycles of total length $k$ from the cycle
decomposition of the invertible part of  $\si$.

In the next theorem, the value of an indecomposable character of the
infinite symmetric semigroup at an element $\si\in R_n$ is presented
as a linear combination of the values of the corresponding Thoma
character at each of the elements of the disjoint union
$\bigsqcup_k M_k(\si)$ with coefficients depending only on the
parameter $\delta$.

\begin{thm}[A formula for the characters]
Let $\chi^{R_\infty}_{\alpha,\beta,\gamma,\delta}\equiv
\chi_{\alpha,\beta,\delta}$ be an indecomposable character of the
algebra ${\mathbb{C}}[R_\infty]$, $\chi^{S_\infty}_{\alpha,\gamma,\beta}\equiv
\chi_{\alpha,\beta}$ be the corresponding indecomposable character of
the algebra ${\mathbb{C}}[S_\infty]$, and $\si\in
R_\infty$ be a reduced element. Then
$$
\chi_{\alpha,\beta,\delta} (\si) = \sum_{k=0}^{n{\si}} \bigg( \delta^{n(\si)-k} (1-\delta)^k \cdot
\sum_{\tilde\si\in M_k(\si)} \chi_{\alpha,\beta}(\tilde\si) \bigg).
$$
\end{thm}

\begin{proof}[Proof]
By Theorem~\ref{thm:approx}, there exists a path
$\{(f, \la_f)\}_{f=0}^\infty$ such that
$$
\chi_{\alpha,\beta,\delta} (\si) = \lim_{f\to\infty} \frac{\chi^*_{(f,\la_f)}(\si)} {\dim
(f,\la_f)}\,.
$$
Recall that an element $\si\in R_n$ is regarded as an element of the
semigroup $R_f$ that is identity on the subset
$\{n+1, \dots, f\}$. By Theorem \ref{thm:chi_R-n}, in order to compute the
character
$\chi^*_{(f,\la_f)}(\si)$, it suffices to describe subsets of size
$|\la_f|$ in the set
$\{1,\dots,f\}$ fixed under the action of the element $\si\in R_f$.
In order to completely describe such subsets, it suffices to associate
with every fixed subset of size $k$ in the set
$\{1,\dots,n\}$ all possible subsets of
$|\la_f|-k$ fixed points in the set $\{n+1,\dots,f\}$.  Thus
$$
\chi^*_{(f,\la_f)}(\si) = \sum_k \bigg( \binom{f-n}{|\la_f|-k} \cdot \sum_{\tilde\si\in M_k(\si)}
\chi_{\la_f} (\tilde\si) \bigg).
$$
By Claim~1 of Lemma~\ref{lem:paths-direct-product},
\begin{align}
\chi_{\alpha,\beta,\delta} (\si) &= \lim_{f\to\infty} \frac{\sum_k \big( \binom{f-n}{|\la_f|-k}
\cdot \sum_{\tilde\si} \chi_{\la_f} (\tilde\si) \big) } {\dim(f,|\la_f|) \cdot \dim(\la_f) }\notag\\
&=\sum_k \bigg( \lim_{f\to\infty} \frac{ \binom{f-n}{|\la_f|-k} } {\dim(f,|\la_f|) } \cdot
\sum_{\tilde\si}\lim_{f\to\infty}\frac{\chi_{\la_f}(\tilde\si)}{\dim(\la_f)}\bigg).
\label{eqn:chi^*_lim_frac}
\end{align}
According to Corollary~\ref{cor:supp-al,be,p} and Theorem~\ref{thm:approx}
applied to the infinite symmetric group $S_\infty$, each of the
summands in the right factor in the right-hand side of
\eqref{eqn:chi^*_lim_frac} tends to the corresponding value of the
character $\chi_{\alpha,\beta}$. Besides, by Corollary
\ref{cor:supp-al,be,p},
$\lim |\la_f|/f = \delta$, whence
$$
\lim_{f\to\infty} \frac{\binom{f-n}{|\la_f|-k}} {\dim(f,|\la_f|)} = \delta^{n-k}(1-\delta)^k,
$$
and this completes the proof.\qed
\end{proof}

\begin{cor}
For an arbitrary element $\si\in R_n\subset R_\infty$,
$$
\chi_{\alpha,\beta,\delta} (\si) = \sum_{k=0}^n \bigg( \delta^{n-k} (1-\delta)^k \cdot
\sum_{\tilde\si\in M_k(\si)} \chi_{\alpha,\beta}(\tilde\si) \bigg).
$$
\end{cor}

\begin{cor}\label{cor:chi|S_infty}
The restriction of a character $\chi_{\alpha,\beta,\delta}$ of the algebra
${\mathbb{C}}(R_\infty)$ to ${\mathbb{C}}(S_\infty)$ is equal to
$\chi_{\alpha', \beta'}$, where $\alpha'_1 = \delta$, $\alpha'_i =
(1-\delta)\alpha_{i-1}$ for $i>1$ and
$\beta'=(1-\delta)\beta$.
\end{cor}

\begin{proof}[Proof]
We will verify the assertion in the case
$\beta = 0$. Let $\alpha'_1 = \delta$, $\alpha'_i =
(1-\delta)\alpha_{i-1}$ for $i>1$, and $\si\in S_n$. Then
$$
\chi^{S_\infty}_{\alpha', 0}(\si) = \prod_\gamma \bigg((1-\delta)^{k_\gamma}\cdot
\sum_i \alpha_i^{k_\gamma}
+ \delta^{k_\gamma} \bigg),
$$
where the product is taken over all minimal cycles
$\gamma$ in the cycle decomposition of the element
$\si$ and $k_\gamma$ are the lengths of these cycles. Expanding the
product, we obtain
$$
\chi^{S_\infty}_{\alpha', 0}(\si) = \sum_k \sum_{\tilde\si\in M_k(\si)} \bigg( (1-\delta)^k
\delta^{n-k} \cdot \prod_\gamma\bigg(\sum_i \alpha_i^{k_\gamma}\bigg) \bigg),
$$
where the internal product is taken over all minimal cycles
$\gamma$ of the subcollection $\tilde\si$. Writing the last equation
in the form
$$
\chi^{S_\infty}_{\alpha', 0}(\si) = \sum_k \biggl( \delta^{n-k} (1-\delta)^k \cdot
\sum_{\tilde\si\in M_k(\si)} \chi^{S_\infty}_{\alpha, 0}(\tilde\si) \biggr) =
\chi^{R_\infty}_{\alpha', 0, \delta}(\si),
$$
we obtain the desired assertion.\qed
\end{proof}

\begin{rem}
In the previous corollary, the parameters
$\alpha$ and $\beta$ are not symmetric, despite the fact that in the graph
$\wt{\mathbb{Y}}$ the symmetry is present. The reason is as follows: under the
embedding of the group  $S_n$ into the semigroup
$R_n$, the restriction of an irreducible representation of
$R_n$ to $S_n$ is the representation induced from a representation of the
subgroup $S_r\times S_{n-r}\subset S_n$ that is trivial on the second factor,
see Remark~\ref{rem:ind-repres}. Hence the operation of restricting a
representation does not commute with the involution
(see Remark~\ref{rem:involution}), which breaks the symmetry between
the parameters
$\alpha$ and $\beta$.
\end{rem}

\subsection{Realization of representations}
We turn our attention to the case where
$\sum_i \alpha_i=1$, i.e.,
$\beta_i=0$ for all $i$. Consider a measure on  $\mathbb{N}$ of the form
$\mu_\alpha(i)=\alpha_i$, the set of sequences
$\mathcal{X}=\prod\mathbb{N}$ equipped with the measure
$m_\alpha=\prod\mu_\alpha$, and the set
$\wt{\mathcal{X}}$ of pairs of sequences coinciding from some point on. In the
space $L^2(\wt{\mathcal{X}}, m_\alpha)$ we can realize the
representation of the symmetric group
$S_\infty$ corresponding to the Thoma parameters $(\alpha, 0)$, see~\cite{VK_factor},
\cite{VershikTsilevich-Realizations}.

\begin{thm}\label{thm:realization}
The realization of the representation of the group
$S_\infty$ cor\-re\-spond\-ing to the parameters $(\alpha', 0)$, where
$\alpha'$  is defined in Corollary~\ref{cor:chi|S_infty},  in the space of functions
$L^2(\wt{\mathcal{X}}, m_{\alpha'})$ can be
extended to a realization of the representation of the semigroup
$R_\infty$ corresponding to the parameters $( \alpha, 0, \delta)$.
\end{thm}

\begin{proof}[Proof]
Define the action of the projection $p_1$ from Theorem~\ref{thm:gens&rels_Popova}
as follows: it maps every sequence
$(a_1,a_2,a_3,\dots)\in\mathcal{X}$ to the sequence
$(1,a_2,a_3,\dots)\in\mathcal{X}$. The relations from
Theorem~\ref{thm:gens&rels_Popova} are obviously satisfied.

Thus it suffices to check that introducing an additional projection
does not lead beyond the space of the representation.
But, as shown in
\cite{Vershik-realizations}, the space of the factor representation of
the symmetric group $S_\infty$ coincides with the whole space
$L^2(\wt{\mathcal{X}}, m_{\alpha'})$, which completes the proof. \qed
\end{proof}

\begin{cor}
In terms of the realization described above, one can give a short
formula for the characters of $R_\infty$,
similar to the formula for the characters of the symmetric group
(cf.~\cite{VK_factor}), which expresses the value of a character
at an element $\si$ as the measure of the set of fixed points of $\si$; namely,
$$
\chi_{\alpha,0,\delta}(\si) = m_{\alpha'}(\{x:\si(x)=x\}),
$$
where $\alpha'$ is defined in Corollary~\ref{cor:chi|S_infty}. See also~\cite{Vershik-realizations}.
\end{cor}

\section{Appendix. General information on finite inverse semigroups}

In this section, we mainly follow the monograph
\cite{CliffordPreston} and the paper \cite{Vershik_bialgebga}.

\subsection{The definition of an inverse semigroup}

\begin{thm}\label{thm:def_inverse-semigroup}
The following two conditions on a semigroup $S$ are equivalent:

(1) for every $a\in S$ there exists $x\in S$ such that $axa = a$, and
any two idempotents of  $S$ commute;

(2) every principal left ideal and every principal right ideal of
$S$ is generated by a unique idempotent;

(3) for every $a\in S$ there exists a unique $x\in S$ such that $axa =
a$ and $xax = x$.
\end{thm}

A semigroup satisfying the conditions of Theorem~\ref{thm:def_inverse-semigroup}
is called an
\textit{inverse semi\-group}. One says that the elements $a$ and $x$ from
condition~(1) of the theorem are
\textit{inverse} to each other; sometimes, this is denoted as
$x = a^{-1}$. Note that
$(ab)^{-1} = b^{-1}a^{-1}$ for any $a,b\in S$.

Let us prove that the symmetric inverse semigroup is an inverse
semigroup. Given a partial map $\si\in R_n$ that acts from a subset
$X\subset \{1,\dots,n\}$ to a subset
$Y\subset \{1,\dots,n\}$, we construct the map $\si^{-1}$ from $Y$ to
$X$ inverse to $\si$ in the ordinary sense, i.e., for
$y\in Y$ and $x\in X$ we set $\si^{-1}(y) = x$ if $\si(x) = y$. The
elements $\si$ and
$\si^{-1}$ are obviously inverse to each other. Besides, the idempotents of
the symmetric inverse semigroup are exactly those maps that send some
subset $X\subset \{1,\dots,n\}$ to itself and are not defined on
$\{1,\dots,n\} \backslash X$. Therefore, any two idempotents commute,
and the semigroup is inverse by Theorem~\ref{thm:def_inverse-semigroup}.

\subsection{An analog of Cayley's theorem} 
Vagner \cite{Vagner} and Preston \cite{Preston} proved for inverse semigroups an analog of
Cayley's theorem for groups.

\begin{thm}
An arbitrary inverse semigroup $S$ is isomorphic to an inverse
subsemigroup of the symmetric inverse semigroup of all one-to-one
partial trans\-for\-ma\-tions of the set $S$.
\end{thm}

The proof is much more difficult than in the group case, and we do
not reproduce it (see~\cite{CliffordPreston}). Note that the theorem
holds both for finite and infinite inverse semigroups.

\subsection{The semisimplicity of the semigroup algebra} 
Given an arbitrary finite semigroup $S$ and a field
$F$, one can consider the semigroup algebra
$F[S]$ of $S$ over $F$. The elements of $S$ form a basis in $F[S]$, and
the multiplication law for these basis elements coincides with the
multiplication law in $S$. Necessary and sufficient conditions for the
semisimplicity of the semigroup algebra $F[S]$ of a finite inverse semigroup $S$
were obtained independently by Munn
\cite{Munn_inverse=semisimple} and Oganesyan
\cite{Oganesyan}.

\begin{thm}\label{thm:inverse_semisimple}
The semigroup algebra $F[S]$ of a finite inverse semigroup $S$ over a
field $K$ is semisimple if and only if the characteristic of $K$
is zero or a prime that does not divide the order of any subgroup in $S$.
\end{thm}

\subsection{Involutive bialgebras and semigroup algebras of inverse semi\-gro\-ups}
A \textit{bialgebra} (see~\cite{Kassel}) is a vector space over the
field ${\mathbb{C}}$ equipped with compatible structures of a unital associative
algebra and a counital coassociative coalgebra.
Name\-ly, the following equivalent conditions are satisfied:

(1) the comultiplication and the counit are homomorphisms of the
corresponding algebras;

(2) the multiplication and the unit are homomorphisms of the
corresponding coalgebras.

Let us also introduce the notion of a
\textit{weakened bialgebra} for the case where
the multiplication and comultiplication are homomorphisms, but there
is no condition on the unit and counit.

The group algebra of a finite group with the convolution
multiplication
and diagonal comultiplication is a cocommutative bialgebra (and even a
Hopf algebra). It is well known  (see \cite{Kassel}) that the
semigroup algebra of every finite semigroup with identity (monoid) is
also a cocommutative bialgebra with the natural definition of the operations.

An involution of an algebra is a second-order antilinear
antiautomorphism of this algebra; a second-order antilinear
antiautomorphism of a coalgebra is called a coinvolution. A bialgebra
equipped with an involution and a coinvolution is called an
\textit{involutive bialgebra}, or a bialgebra with involution, if the
multiplication commutes with the coinvolution and the comultiplication
commutes with the involution.

In \cite{Vershik_bialgebga} it was shown that the class of finite
inverse semigroups generates exactly the class of involutive
semisimple bialgebras.

\begin{thm}
The semigroup algebra of a finite inverse semigroup is a semi\-sim\-ple
cocommutative  involutive algebra. Analogously, the dual semigroup
algebra ${\mathbb{C}}[S]$ of a finite inverse semigroup $S$ with identity is a
commutative involutive bialgebra. Conversely, every finite-dimensional
semisimple cocommutative (in the dual case, commutative) involutive
bialgebra is isomorphic (as an involutive bialgebra) to the semigroup
algebra (respectively, dual semigroup
algebra) of a finite inverse semigroup with identity.

For inverse semigroups without identity, the semigroup bialgebra is a
weakened bialgebra (the counit is not a homomorphism).
\end{thm}

\bigskip
Translated by N.~V.~Tsilevich.

%


\begin{thebibliography}{10}

\bibitem{Vagner}
V.~V.~Vagner
\newblock Generalized groups
\newblock {\em Doklady Akad. Nauk SSSR (N.S.)}, 84:24--43, 1952.

\bibitem{Vershik_bialgebga}
A.~M.~Vershik
\newblock Krein's duality, positive 2-algebras, and the dilation of comultiplications
\newblock {\em Funct. Anal. Appl.}, 41(2):99--114, 2007.

\bibitem{Vershik-realizations}
A.~M.~Vershik
\newblock Nonfree actions of groups and the theory of characters,
\newblock {\em in preparation}.

\bibitem{VK}
A.~M.~Vershik and S.~V.~Kerov
\newblock Asymptotic theory of the characters of a symmetric group
\newblock {\em Funktsional. Anal. i Prilozhen.}, 15(4):15--27, 1981.

\bibitem{VK_factor}
A.~M.~Vershik and S.~V.~Kerov
\newblock Characters and factor representations of the infinite symmetric group
\newblock {\em Dokl. Akad. Nauk SSSR}, 257(5):1037--1040, 1981.

\bibitem{VK_survey2}
A.~M.~Vershik and S.~V.~Kerov
\newblock Locally semisimple algebras. Combinatorial theory and the $K_0$-functor
\newblock {\em Itogi Nauki i Tekhniki, Ser. Sovrem. Probl. Mat.}, VINITI, 26:3--56, 1985.

\bibitem{VershikNikitin}
A.~M.~Vershik and P.~P.~Nikitin
\newblock Traces on infinite-dimensional Brauer algebras
\newblock {\em Funct. Anal. Appl.}, 40(3):165--172, 2006.

\bibitem{Kassel}
C.~Kassel
\newblock {\em Quantum groups}.
\newblock Springer-Verlag, New York, 1995.

\bibitem{CliffordPreston}
A.~H.~Clifford and G.~B.~Preston
\newblock {\em The algebraic theory of semigroups}.
\newblock Amer. Math. Soc., Providence, R.I., 1961.

%
\bibitem{Oganesyan}
V.~A.~Oganesyan
\newblock On the semisimplicity of a system algebra
\newblock {\em Akad. Nauk Armyan. SSR Dokl.}, 21:145--147, 1955.

\bibitem{Popova}
L.~I.~Popova
\newblock Defining relations for some subgroups of partial transformations of a finite set
\newblock {\em Uch. Zapiski Leningr. Gos. Ped. Inst. im. A.~I.~Gertsena}, 218:191--212, 1961.

\bibitem{Halverson}
T.~Halverson
\newblock Representations of the q-rook monoid.
\newblock {\em J.~Algebra}, 273(1):227-251, 2004.

\bibitem{Munn_symm-inverse-semigroup}
W.~D.~Munn
\newblock The characters of the symmetric inverse semigroup.
\newblock {\em Proc. Camb. Phil. Soc.}, 53(1):13--18, 1957.

\bibitem{Munn_inverse=semisimple}
W.~D.~Munn
\newblock On semigroup algebras.
\newblock {\em Proc. Cambridge Phil. Soc.}, 51:1--15, 1955.

\bibitem{Olshansky}
G.~Olshansky
\newblock Unitary representations of the infinite symmetric group: a semigroup approach.
\newblock {\em Representations of Lie groups and Lie algebras}, Acad{\'e}miai Kiad{\'o}, Budapest, 1985, pp. 181--197.

\bibitem{Preston}
G.~B.~Preston
\newblock Representations of inverse semigroups.
\newblock {\em J. London Math. Soc.}, 29:411--419, 1954.

\bibitem{Solomon}
L.~Solomon
\newblock Representations of the rook monoid,
\newblock {\em J.~Algebra}, 256(2):309--342, 2002.

\bibitem{Thoma}
E.~Thoma
\newblock Die unzerlegbaren, positiv-definiten Klassenfunktionen der abz{\"a}hlbar unendlichen
symmetrischen Gruppe
\newblock {\em Math. Zeitschr.}, 85(1):40--61, 1964.

\bibitem{Vershinin}
V.~V.~Vershinin
\newblock On the inverse braid monoid.
\newblock {\em Topology Appl.}, 156(6):1153-1166.

\bibitem{VK_survey}
A.~M.~Vershik, S.~V.~Kerov
\newblock The Grothendieck group of the infinite symmetric group and symmetric Functions (with the elements of the theory $K_0$-functor
of AF-algebras)
\newblock {\em Adv. Stud. Contemp. Math.}, Gordon and Breach, 7:39--118, 1990.

\bibitem{VershikTsilevich-Realizations}
A.~M.~Vershik, N.~V.~Tsilevich
\newblock On different models of representations of the infinite symmetric group.
\newblock {\em Adv. Appl. Math.}, 37:526--540, 2006.

\end{thebibliography}
\end{document}